\documentclass[12pt]{amsart}
\usepackage{amsmath}
\usepackage{amssymb}
\usepackage{latexsym}

\newtheorem{theorem}{Theorem}[section]

\newtheorem{lemma}[theorem]{Lemma}

\newtheorem{prop}[theorem]{Proposition}
\newtheorem{defn}[theorem]{Definition}

\newenvironment{proof*}{\vskip 2mm\noindent {}}{\hfill $\Box$ \vskip 2mm}
\numberwithin{equation}{section}
\newcommand{\C}{{\mathbb{C}}}
\newcommand{\D}{{\mathbb{D}}}

\newcommand{\N}{{\mathbb{N}}}

\newcommand{\I}{{\mathcal{I}}}

\renewcommand{\O}{{\mathcal{O}}}
\newcommand{\eps}{\varepsilon}

\begin{document}

\title[Limit of four-pole Green functions]{Limit of Green functions and ideals, the case 
of four poles
\footnote{2010 {\it Mathematics Subject Classification} 32U35, 32A27}
}

\author{Duong Quang Hai, Pascal J. Thomas}

\address{Universit\'e de Toulouse\\ UPS, INSA, UT1, UTM \\
Institut de Math\'e\-matiques de Toulouse\\
F-31062 Toulouse, France} 
\email{quanghai@math.univ-toulouse.fr, pascal.thomas@math.univ-toulouse.fr}

\begin{abstract}
We study the limits of pluricomplex Green functions with four poles tending to
the origin in a hyperconvex domain, and the (related) limits of the ideals of
holomorphic functions vanishing on those points. 
Taking subsequences, we always assume that the directions defined by pairs
of points stabilize as they tend to $0$. We prove that in a generic
case, the limit of the Green functions is always the same, while the limits
of ideals are distinct (in contrast to the three point case).  We also 
study some exceptional cases, where only the limits of ideals are determined.
In order to do this, we establish a useful result linking the length of
the upper or lower limits of a family of ideals, and its convergence.
\end{abstract}

\keywords{pluricomplex Green function, complex Monge-Amp\`ere equation, ideals of holomorphic functions}

\thanks{This work, in a different form, is part of the first author's Ph.D. dissertation \cite{Hai}, defended at the Universit\'e Paul Sabatier, Toulouse, July 8th, 2013.}

\maketitle

\section{Introduction}
The definition of multipole pluricomplex Green functions with logarithmic singularities,
in the wake of Lempert's seminal work \cite{Lem},
was motivated by the nonlinearity of the complex Monge-Amp\`ere equation,
and generalizations of the Schwarz Lemma, see
e.g. Demailly \cite{De1},  \cite{Za},  Lelong \cite{Lel}.

Sometimes it is useful to study the limit case where poles tend to each
other \cite{Th}, an analogue of multiple zeroes for holomorphic functions, 
and this leads naturally to the more general notion of the Green
function of an ideal of holomorphic functions:
\begin{defn}\cite{Ra-Si}
\label{greenideal}
Let $\Omega$ be a hyperconvex bounded domain
in $\C^n$, $\mathcal O(\Omega)$ the space of holomorphic functions on this domain.

Let $\mathcal I$ be an ideal of $\mathcal O(\Omega)$, and $\psi_j$ its generators. Then
$$
G_{\mathcal I}^\Omega (z) := \sup \big\{ u(z) : u \in PSH_-(\Omega), 
u(z) \le \max_j \log |\psi_{j}| + O(1) \big\}.
$$
\end{defn}
Note that the condition is meaningful only near $a \in V(\mathcal I):=
\{ p \in \Omega: f(p)=0, \forall f \in \mathcal I\}$. Since the domain is pseudoconvex,
 there are finitely many global generators $\psi_j \in \mathcal O(\Omega)$ such that
for any $f\in \mathcal I$, there exists $h_j \in \mathcal O(\Omega)$ such that 
$f = \sum_j h_j \psi_j$, see e.g. \cite[Theorem 7.2.9, p. 190]{Ho}. 

In the special case when $S$ is a finite set in $\Omega$ and
$\mathcal I=\mathcal I(S)$, the ideal of all functions vanishing on the set $S$
(which we sometimes call point-based ideal),
this 
reduces to a pluricomplex Green function with logarithmic singularities;
we write $G_{\mathcal I(S)}=G_S$. 

We want to study the limit of $G_{S_\eps}$
when $S_\eps$ is a set of points tending to the origin, and relate this to the
limit of the ideals $\I(S_\eps)$. It is a consequence of \cite{Ra-Th} that 
if convergence of those Green functions takes place in the (relatively weak)
sense of $L^1_{\mbox{loc}}(\overline \Omega)$, 
then that convergence is actually uniform on compacta of $\overline \Omega \setminus \{0\}$,
so it will be understood that all convergence results are in this sense. 

The case of $3$ poles in dimension $n=2$
was worked out in
\cite[Theorem 1.12, (i)]{MRST}; a remaining subcase of that study was finally
settled in \cite{Du-Th}. 

In the present paper, we explore the case of $4$ points tending to the origin in
$\C^2$.  Unlike in the three-point case, where the limit ideal was generically
$\frak M_0^2$ and the limits of the Green functions depended on the directions
along which the points tended to $0$, 
here we will see that, generically (in a sense 
to be made precise), $\lim G_{\mathcal I_\eps} = G_{\lim \mathcal I_\eps}$,
and that this limit is the same, namely, 
$\lim G_{\mathcal I_\eps} = 2 \max (\log |z_1|, \log |z_2|)+O(1)$
(Theorem \ref{theo_ICU_chap_4}), whereas the
limit ideals very much depend on the directions of convergence to $0$. 

Some singular cases are studied in Theorems \ref{theo_02_chap_04} and
\ref{pro11122012_01}, although here we mostly compute
limits of ideals, the Green functions of which cannot coincide with the
limit of our Green functions because of Theorem \ref{thmiff} below.  The
results of \cite{Ra-Th} are used to yield some estimates of the Green
functions in those cases, but the complete answer is not known.

In order to obtain those results, we establish Theorem \ref{limsupinf},
 an auxiliary result about convergence of
ideals which shortens the proofs, and should be of independent interest. 

\section{Statement of the results}

\subsection{Notations}

As usual, $\frak M_0:= \mathcal I(\{(0,0)\})$ stands for the maximal ideal at $(0,0)$,
and $\frak M_0^2, \frak M_0^3 \dots$ for its successive powers.  For an ideal $\mathcal I \subset
\mathcal O(\Omega)$, its \emph{length} (or co-length) is 
$\ell (\mathcal I) := \dim O(\Omega)/ \dim (\mathcal I)$.  For instance, 
$\ell (\frak M_0^k) = \frac12 k(k+1)$. 

We consider $S_\varepsilon := \{a_k^\varepsilon,1\le k \le 4\} \subset \Omega$, for $\eps \in \C$,
$\mathcal I_\eps := \mathcal I(S_\eps)$.

In general we should consider $A \subset \C$ such that $0 \in \Bar{A}\backslash A$ and study
limits along $A$; quite often we will use some compactness 
to ensure convergence and pass to a subsequence
included in $A$.  For simplicity, we will just write
 $\underset{\varepsilon \rightarrow 0}{\lim}$ or $\lim_\eps$ instead of 
 $\underset{\varepsilon \rightarrow 0, \varepsilon \in A}{\lim}$.

We will write several sufficient conditions about convergence of ideals and Green functions
in terms of the asymptotic directions defined by pairs of poles:
$$v_{ij}^\varepsilon := [a_j^\varepsilon - a_i^\varepsilon] \in \mathbb{P}^1\C,
$$
where $[\cdot]$ denotes the class in $\mathbb{P}^1\C$ of an element of $\C^2 \setminus \{(0,0)\}$.
Since $\mathbb{P}^1\C$ is compact, by restricting to an appropriate subsequence
we assume
 $v_{ij} = \underset{\varepsilon \rightarrow 0}{\lim} v_{ij}^\varepsilon \in \mathbb{P}^1\C$, for $1 \leqslant i < j \leqslant 4$.  When such convergence does not occur as $\eps\to 0$ in an
 unrestricted fashion, one may consider the (possible) limits obtained from ``convergent" 
 subsequences, and conclude about global convergence by examining whether the partial limits
 coincide or not. 
 
Let
\begin{equation*}
\mathcal{D}^\varepsilon = \mathcal{D}(S_\varepsilon):= \{v_{ij}^\varepsilon \in \mathbb{P}^1\C, 1 \leqslant i < j \leqslant 4\},
\quad
\mathcal{D} := \{v_{ij} \in \mathbb{P}^1\C, 1 \leqslant i < j \leqslant 4\}.
\end{equation*}

Given a subset $\tilde S_\eps \subset S_\eps$, we can define $\tilde{\mathcal D}_\eps$ and
$\tilde{\mathcal D}$ in a similar manner.

\subsection{The generic $4$-pole case}

\begin{theorem}\label{theo_ICU_chap_4} 
Let $S_\varepsilon$ satisfy
\begin{equation}\label{condition21112011}
\forall \tilde S_\eps \subset S_\eps \mbox{ with } \#\tilde S_\eps=3\mbox{, then }
\#\Tilde{\mathcal{D}} \geqslant 2.
\end{equation}
and 
\begin{equation}\label{condition20092012}
\forall k \in \{1, 2, 3, 4\}, \quad 
\#\big\{v_{km} \in \mathbb{P}^1\C : m \in \{1, 2, 3, 4\} \backslash \{k\}\big\} \geqslant 2,
\end{equation}
then there exists $\lim_\eps \mathcal{I}_\varepsilon = \mathcal I$, with 
$\mathfrak{M}_0^3 \subset \mathcal I \subset \mathfrak{M}_0^2$ and $ \ell(\mathcal I) = 4$; and
$\lim_\eps G_\varepsilon = G_J = 2  \max (\log |z_1|, \log |z_2|)+O(1)$
depends only on $\Omega$ and not on $\mathcal I$.
\end{theorem}

\subsection{Some singular cases}

We will see how things change when we give up the second condition in
Theorem \ref{theo_ICU_chap_4}.

\begin{theorem}\label{theo_02_chap_04} Suppose that
 $S_\varepsilon$ verifies condition \eqref{condition21112011},
 and
\begin{equation}\label{eq_15112012_02}
\exists i \in I := \{1, 2, 3, 4\} \mbox{ s.t. }
\# \big\{v_{ij} \in \mathbb{C}\mathbb{P}^1 : j \in I \backslash \{i\} \big\} = 1,
\end{equation}
then, after a linear change of variables,
 $\underset{\varepsilon \to 0}{\lim}\; \mathcal{I}\big(S_\varepsilon\big) = \mathcal{I}_0 := \left\langle z_1z_2, z_2^2, z_1^3\right\rangle$, 
and 
\begin{equation*}
\liminf_\eps G_\eps \ge
G_{\mathcal{I}_0}(z) = \max\big\{\log |z_1z_2|, 2\log |z_2|, 3\log |z_1|\big\} + O(1),
\end{equation*}
but there is no equality.
\end{theorem}

If the situation becomes even more singular, we can have more diverse limits for the
ideals.

\begin{theorem}\label{pro11122012_01} 
Suppose there exist a $3$ point subset $\Tilde{S}_\varepsilon \subset S_\varepsilon$ 
such that  $\# \Tilde{\mathcal{D}} = 1$.
Then
\begin{enumerate}
\item
If $\# \mathcal{D} \geqslant 3$,  then, after an appropriate linear change of variables,
$\underset{\varepsilon \to 0}{\lim}\; \mathcal{I}\big(S_\varepsilon\big) = 
\mathcal{I}_0$.
\item
If $\# \mathcal{D} = 2$, then, after passing to a subsequence and
an appropriate linear change of variables,
$\underset{\varepsilon \to 0}{\lim}\; \mathcal{I}\big(S_\varepsilon\big) = 
\mathcal{I}_0$ or $=\mathcal{J}_0:=  \left\langle z_1z_2, z_1^2 + k z_2^2, z_1^3\right\rangle$,
for some $k \in \C \setminus \{0\}$.
\end{enumerate}
\end{theorem}

We suspect that the Green functions do admit a limit, but we haven't been able to determine it.

\subsection{Upper and lower limits of ideals}

We now formalize the notion of convergence of ideals using upper and lower limits.

\begin{defn}{\rm \cite{MRST}}
\label{convid}
\begin{enumerate}
\item[(i)]
$\liminf\limits_{A\ni\eps \to 0}\I_\eps$ is the ideal consisting of all $f\in \mathcal O(\Omega)$ such that $f_\eps\to f$ locally
uniformly on $\Omega$, as $\eps\to 0$, where $f_\eps \in \I_\eps$.\\
\item[(ii)]
$\limsup\limits_{A\ni\eps \to 0}\I_\eps$ is the ideal
of $\mathcal O(\Omega)$ generated by all functions $f$ such that
$f_j\to f$ locally uniformly, as $j\to \infty$, for some
sequence $\eps_j\to 0$ in $A$ and $f_j\in \I_{\eps_j}$.\\
\item[(iii)]
If the two limits are equal, we say that the family
$\I_\eps$ {\rm converges} and write
$\lim\limits_{A\ni\eps \to 0}\I_\eps$ for
the common value of the upper and lower limits.
\end{enumerate}
\end{defn}

This last notion of convergence is equivalent to convergence in the topology of the Douady space
\cite[Section 3]{MRST}.  Clearly, $\liminf_\eps \I_\eps \subset \limsup_\eps \I_\eps$
and so $\ell(\liminf_\eps \I_\eps) \ge \ell(\limsup_\eps \I_\eps)$.  It also 
follows from \cite[Lemmas 2.1 and 2.2]{MRST} that $\ell(\limsup_\eps \I_\eps) \le \limsup \ell (I_\eps)$
and $\ell(\liminf_\eps \I_\eps) \le \liminf \ell (I_\eps)$.

\begin{theorem}
\label{limsupinf}
Let $\I_\eps$ be a family of ideals based on $N$ distinct points, so that $\ell(\I_\eps)=N$,
for any $\eps$.  
\begin{enumerate}
\item[(i)]
Let $\I := \limsup_\eps \I_\eps$. If $\ell(\I)\ge N$ (or equivalently $=N$), then
$\lim_\eps \I_\eps = \I$. 
\item[(ii)]
Let $\I := \liminf_\eps \I_\eps$. If $\ell(\I)\le N$ (or equivalently $=N$), then
$\lim_\eps \I_\eps = \I$. 
\end{enumerate}
\end{theorem}

\section{Proof of Theorem \ref{limsupinf}}

We will proceed by reducing everything to upper and lower limits of subspaces
of a single finite-dimensional vector space. 

We use multiindex notation, in particular if $\alpha, \beta \in \N^n$, $\alpha \le \beta$
means $\alpha_j \le \beta_j$, $1\le j \le n$ (and the analogous definition for ``$<$").

Let $\pi_j$ denote the projection to the $j$-th coordinate axis. Passing
to a subsequence if needed,
$N_j:=\#\pi_j(\{a_1^\eps, \dots, a_N^\eps\})$ is independent of $\eps$.
Let $\mathcal N:= (N_1, \dots, N_n)$ and
$$
P_\eps := \pi_1(\{a_1^\eps, \dots, a_N^\eps\}) \times \cdots
\times \pi_n(\{a_1^\eps, \dots, a_N^\eps\})  \mbox{ (cartesian product)}
$$
As in \cite[Section 2]{MRST}, 
we now define a simpler sequence of ideals contained in each
 $\I_\eps=\I( \{a_1^\eps, \dots, a_N^\eps\})$. 
Let $\mathcal J_\eps := \I (P_\eps)$.  It is easy to see that $d:= \ell (\mathcal J_\eps)
= \#P_\eps = \prod_{j=1}^n N_j \le N^n$, and \cite[Lemma 2.3]{MRST} gives
%\begin{multline*}
$$
\lim_{\eps\to0} \mathcal J_\eps = \mathcal J
:= \langle z_1^{N_1}, \dots, z_n^{N_n} \rangle
= \left\{ f \in \O(\Omega) : 
\frac{\partial^{\alpha} f}{\partial z^\alpha} : \alpha < \mathcal N
%\frac{\partial^{k_1+\dots+k_n} f}{\partial z_1^{k_1} \dots \partial z_n^{k_n}}(0)=0,
%1 \le j \le n, 0\le k_j \le N_j-1
\right\}. 
$$
%\end{multline*}
{\bf Claim.} $\mathcal{O} / \mathcal{J}_\varepsilon \cong \mathbb{C}^d \cong \mathcal{O} / \mathcal{J}$. 

Indeed, denote by $b_j^{i, \varepsilon}$ the elements of $\pi_j(\{a_1^\eps, \dots, a_N^\eps\})$.
For $ \alpha \le \mathcal N$, set $\Psi_\alpha(z) = z^\alpha$, and for
$\zeta \in \mathbb{C}$, $1 \leqslant j \leqslant n$, $0\leqslant k  \leqslant N_j - 1$,
\begin{equation*}
\varphi_{k,j}^\varepsilon\big(\zeta\big) := \prod_{i=1}^k
(\zeta - b_j^{i, \varepsilon})
 \end{equation*}

Let 
\begin{equation*}
\Psi_\alpha^\varepsilon(z) :=  \prod_{j=1}^n \varphi_{\alpha_j,j}^\varepsilon\big(z_j\big).
\end{equation*}
Since all the $b_j^{i, \varepsilon}$ tend to $0$, it is easy to see that
for $\eps$ small enough (including $\eps=0$) 
the system $\{\Psi_\alpha^\varepsilon,  \alpha < \mathcal N\}$ is linearly independent.

Let $[\cdot]_\eps$ (resp. $[\cdot]$) denote the class of a function in 
$\mathcal{O} / \mathcal{J}_\eps$ (resp. $\mathcal{O} / \mathcal{J}$). 
The natural projection from $\mbox{Span} \{\Psi_\alpha^\varepsilon,  \alpha \le \mathcal\}$
to $\mathcal{O} / \mathcal{J}_\eps$ is injective, thus bijective, and 
$\{[\Psi_\alpha^\varepsilon]_\eps,  \alpha < \mathcal N\}$ is a basis
of $\mathcal{O} / \mathcal{J}_\eps$.  Then the linear map defined by 
$\Phi_\varepsilon \bigg(\big[\Psi_\alpha^\varepsilon\big]_{\mathcal{J}_\varepsilon}\bigg) = \big[\Psi_\alpha\big]$, for $\alpha < \mathcal N$, is the required isomorphism.

\begin{lemma}
\label{pro_chap2_280313} 
Suppose that $\lim_\eps f_\eps =f$, uniformly on compacta of $\Omega$. Then,
in the finite dimensional vector space $\mathcal{O} / \mathcal{J}$, 
$\big\{\Phi_\varepsilon \big([f_\eps]_{\mathcal{J}_\varepsilon}\big)\big\}  \to [f]$ as $\varepsilon \to 0$.
\end{lemma}

\begin{proof}
There is a unique choice of coefficients $c_\alpha^\eps (f) $ such that
 $f_\eps = \sum_{\alpha < \mathcal N} c_\alpha^\eps (f_\eps) \Psi_\alpha^\varepsilon + h_\eps$,
 with $h_\eps \in \mathcal{J}_\varepsilon$.
It will be enough to show that $c_\alpha^\eps (f_\eps) \to c_\alpha (f)$ as $\varepsilon \to 0$,
for each $\alpha$.

By rescaling, we might assume that $\overline \D^n \subset \Omega$.
One can prove by induction on $n$ (or deduce as an easy special case
from the beginning of \cite{Ts}) that 
if $|\eps|$ is small
enough, then 
$$
c_\alpha^\eps (f_\eps) = \frac1{(2i\pi)^n}\int_{(\partial \D)^n} \frac{f_\eps(z_1, \dots, z_n)}
{\Psi_\alpha^\varepsilon(z)} \frac{dz_1}{z_1} \dots \frac{dz_n}{z_n},
$$
and one sees that those integrals converge towards the required limit.
\end{proof}

We define upper and lower limits for families of subspaces in a finite dimensional
vector space $\C^d$ by first choosing a norm on it.  Since they are equivalent,
we may as well choose a euclidean norm, and we do.

Then let $L_\varepsilon$ be 
a family of subspaces of $\mathbb{C}^d$ 
such that $\dim L_\varepsilon = k$, for any $\eps$. 
Let $K_\varepsilon := L_\varepsilon \cap \overline{B}(0;1)$. 
We can define the upper and lower limits of 
   $L_\varepsilon$ 
by $\liminf_\varepsilon L_\varepsilon := Span\big(\lim\inf K_\varepsilon\big)$ 
where $\lim\inf K_\varepsilon$ is taken in the sense of the Hausdorff distance
between compacta (and inclusion as an order relation), and analogously
$\limsup_\varepsilon L_\varepsilon := Span\big(\lim\sup K_\varepsilon\big)$.

\begin{prop}\label{pro080612_02} 
\begin{enumerate}
\item 
$\underset{\varepsilon \rightarrow 0}{\limsup} \,\Phi_\varepsilon\big(\mathcal{I}_\varepsilon / \mathcal{J}_\varepsilon\big) = (\limsup\mathcal{I}_\eps) / \mathcal{J}$,
\item
$\underset{\varepsilon \rightarrow 0}{\liminf} \,\Phi_\varepsilon\big(\mathcal{I}_\varepsilon / \mathcal{J}_\varepsilon\big) = (\liminf \mathcal{I}_\eps) / \mathcal{J}$.
\end{enumerate}
\end{prop}
\begin{proof}
To prove that $\limsup \mathcal I_\eps / \mathcal{J} \subset \limsup\left( \Phi_\varepsilon(\mathcal I_\eps /\mathcal J_\eps)\right)$, it is enough to consider elements $[f]$ where $f$ is in a generating system
of $\limsup \mathcal I_\eps$.  So 
there exist
$(\varepsilon_j)_{j \in \mathbb{Z}_+}, \varepsilon_j \to 0$ as $j \to +\infty$ 
and $f_j \in \mathcal{I}_{\varepsilon_j}$ such that $f_j \to f$ uniformly on compacta of $\Omega$. 
Proposition \ref{pro_chap2_280313} implies that
$\Phi_{\varepsilon_j}\big([f_j]_{\mathcal{J}_{\varepsilon_j}}\big) \to [f]$.

Conversely, take $g \in \mathcal{O} / \mathcal{J}$ such that 
there exists $(\varepsilon_j)_{j \in \mathbb{Z}_+}, \varepsilon_j \to 0$ as $j \to +\infty$ and
$g_j \in \mathcal{I}_{\varepsilon_j}$ such that $\|\Phi_{\varepsilon_j}\big([g_j]_{\mathcal{J}_{\varepsilon_j}}\big) - [g]\| \to 0$ as $j \to +\infty$. Then 
$|C_\alpha^{\varepsilon_j}(g_j) - C_\alpha(g)| \to 0$ for any $\alpha <N$. 
We can write
\begin{equation*}
g(z) = \underset{\alpha \in \Gamma}{\sum} C_\alpha(g) z^\alpha + \sum_{j=1}^n z_j^{N_j}R_j(z) 
\mbox{ and }
\end{equation*}
\begin{equation*}
\big[g_j(z)\big]_{\mathcal{J}_{\varepsilon_j}} = \underset{\alpha \in \Gamma}{\sum} C_\alpha^{\varepsilon_j}(g_j) \big[\Psi_\alpha^{\varepsilon_j}(z)\big]_{\mathcal{J}_{\varepsilon_j}} \in \mathcal{I}_{\varepsilon_j} / \mathcal{J}_{\varepsilon_j}.
\end{equation*}
Set
\begin{equation*}
f_j(z) := \underset{\alpha \in \Gamma}{\sum} C_\alpha^{\varepsilon_j}(g_j) \Psi_\alpha^{\varepsilon_j}(z) + 
\sum_{j=1}^n \prod_{i=1}^{N_j}\big(z_j - b_j^{i,\varepsilon_j}\big)R_j(z).
\end{equation*}
Then $f_j \in \mathcal{I}_{\varepsilon_j}$ and $f_j \to g$ uniformly on compacta of $\Omega$. 

Since the $g$'s as above form a generating system
for $\limsup\left( \Phi_\varepsilon(\mathcal I_\eps /\mathcal J_\eps)\right)$, we are done.

The proof for $\liminf$ is analogous and we omit it.
\end{proof}

The proof of our theorem then reduces to an elementary fact about families
of finite dimensional spaces. 

 \begin{lemma}\label{lem080612_01} 
 Let $(L_\varepsilon)$ be a family of vector subspaces of $\mathbb{C}^d$ 
such that $\dim L_\varepsilon = k \le n$, for any $\varepsilon$.
\begin{enumerate}
\item
  If $\dim(\underset{\varepsilon \rightarrow 0}{\lim\sup}\; L_\varepsilon) = k$, then 
$\underset{\varepsilon \rightarrow 0}{\lim\inf}\; L_\varepsilon = \underset{\varepsilon \rightarrow 0}{\lim\sup}\; L_\varepsilon$.
\item
  If $\dim(\underset{\varepsilon \rightarrow 0}{\lim\inf}\; L_\varepsilon) = k$, then 
$
\underset{\varepsilon \rightarrow 0}{\lim\inf}\; L_\varepsilon = \underset{\varepsilon \rightarrow 0}{\lim\sup}\; L_\varepsilon$.
\end{enumerate}
\end{lemma}

\begin{proof}
(1).  Let $L$ stand for $\limsup L_\eps$. For any $\eta \in (0, \frac12)$, there exists
$\eps_\eta >0$ such that $|\eps| \le \eps_\eta$ implies that $L_\eps \cap \overline{B}(0;1)$
is contained in an $\eta$-neighborhood of $L \cap \overline{B}(0;1)$.  So the orthogonal
projection of $L_\eps \cap \overline{B}(0;1)$ to $L$ must contain at least the ball
$L\cap \overline{B}(0;(1-\eta^2)^{1/2})$, and any point of $L \cap \overline{B}(0;1)$
is a distance at most $\eta + 1-(1-\eta^2)^{1/2}$ from $L_\eps \cap \overline{B}(0;1)$,
so $L \subset \liminf_\eps L_\eps$.

(2). Let $L:= \liminf L_\eps$. If we had $\limsup L_\eps \not \subset L$, then
$\limsup L_\eps \supsetneq L$ and we can pick a unit vector $v \in \limsup L_\eps \cap L^\perp$. 
We can find a sequence $\eps_j \to 0$ and vectors $v_j \to v$, $v_j \in L_{\eps_j}$.
$L_{\eps_j}$ must also contain $k$ vectors $e_1^{\eps_j}, \dots, e_k^{\eps_j}$ close 
to the vectors in an orthonormal basis $e_1, \dots, e_k$ 
of $L$.  For $j$ large enough, the system $e_1^{\eps_j}, \dots, e_k^{\eps_j}, v_j$ will 
have to be linearly independent, which contradicts $\dim L_{\eps_j}=k$.
\end{proof}

\section{Proofs of Theorems \ref{theo_ICU_chap_4}, \ref{theo_02_chap_04} and \ref{pro11122012_01} }

\subsection{Previous results}

\begin{defn}
A (point based) ideal is a \emph{complete intersection ideal} if and only if 
it admits a set of $n$ generators, where $n$ is the dimension of the ambient space.
\end{defn}

The main result of \cite{MRST}, Theorem 1.11, states:

\begin{theorem}
\label{thmiff}
Let  $\mathcal I_\eps= \mathcal I(S_\eps)$, 
where $S_\eps$ is a set of $N$ points all tending to $0$ and
assume that
$\lim_{\eps\to 0} \mathcal I_\eps = \mathcal I$.
Then $(G_{\mathcal I_\eps})$ converges to $G_{\mathcal I}$ locally
uniformly on $\Omega \setminus \{0\}$ if and only if 
$\mathcal I$ is a complete intersection ideal.
\end{theorem}

The following was also defined in \cite{MRST}.

\begin{defn}
\label{defgci} 
The family of ideals $(\mathcal I_\eps)$ satisfies the \emph{Uniform Complete Intersection
Condition} if for any $\eps$, there exists a map $\Psi_0$ and maps $\Psi_\eps$ 
 from a neighborhood
of $\overline \Omega$ to $\C^n$ such that $\Psi_0$  is proper 
from $\Omega$ to $\Psi_0(\Omega)$, and
\begin{enumerate}
\item 
$\{a_j^\eps, 1 \le j \le N\} = \Psi_\eps^{-1}\{0\}$, for all $\eps$;
\item
For all $\eps \neq 0$, $1\le j \le N$ and $z$ in a neighborhood of $a_j^\eps$,
$$
\left| \log \|\Psi_\eps(z) \| - \log \|z - a_j^\eps \|
\right| \le C(\eps) < \infty ;
$$
\item
$\lim_{\eps\to 0}  \Psi_\eps = \Psi=( \Psi^1, \dots,  \Psi^n)$, uniformly on  $\overline \Omega$.
\end{enumerate}
\end{defn}
Notice that the first two conditions imply
$\mathcal I_\eps = \langle  \Psi_\eps^1, \dots,  \Psi_\eps^n \rangle.$

This is \cite[Theorem 1.8]{MRST}:

\begin{theorem}
\label{thmgci}
Let $(\I_\eps)$ be a family of ideals satifying the uniform complete
intersection condition, set $S_\eps=V(\I_\eps)$ and
$\I=\langle \Psi^1,\dots,\Psi^n\rangle$.  Then 
\begin{enumerate}
\item
$\lim\limits_{\eps\to 0} \mathcal I_\eps = \mathcal I$,
\item
$\lim\limits_{\eps\to 0} G_\eps = G_{\mathcal I}$,
and the convergence is locally uniform on $\Omega \setminus \{0\}$.
\end{enumerate}
\end{theorem}

\subsection{Proof of Theorem \ref{theo_ICU_chap_4}}
\ {}

Let $l_{ij}^\varepsilon, 1 \leqslant i < j \leqslant 4$ be the (normalized) equations 
of the lines through $a_i^\varepsilon, a_j^\varepsilon$ and $l_{ij} := \underset{\varepsilon \to 0}{\lim} l_{ij}^\varepsilon, 1 \leqslant i < j \leqslant 4$. Set
\begin{equation*}
\mathcal{L}^\varepsilon := \left\{f_1^\varepsilon := l_{12}^\varepsilon.l_{34}^\varepsilon; f_2^\varepsilon := l_{13}^\varepsilon.l_{24}^\varepsilon; f_3^\varepsilon := l_{14}^\varepsilon.l_{23}^\varepsilon\right\} \subset \mathcal{I}(S_\varepsilon),
\end{equation*}
and $f_j := \underset{\varepsilon \to 0}{\lim}f_j^\varepsilon$, $j = 1, 2, 3$.  

We will prove that under the hypotheses of the theorem, 
there exists $i\neq j \in \{1,2,3\}$ such that if 
$\Psi_0 := \big(f_i, f_j\big)$, then $\Psi_0^{-1}(0) = \{0\}$.  (One can see that
the hypotheses are necessary for this to happen \cite[Remarque 4.1.2, p. 66]{Hai}).  
Then we conclude using Theorem \ref{thmgci}
with $\Psi_\eps:= f_i^\varepsilon f_j^\varepsilon$.  Notice that since $\Psi_0$
is homogeneous of degree $2$ and $\| \Psi_0\|$ is bounded and bounded away from $0$
on the unit sphere, then $\log \| \Psi_0\| = \log \|z\|^2 + O(1)$, and the
same estimate holds for $G_{\mathcal I}$.  An application of the 
generalized maximum principle of Rashkovskii and Sigurdsson \cite[Lemma 4.1]{Ra-Si} 
shows that the limit does not depend on the particular value of $\| \Psi_0\|$: there
is only one maximal plurisubharmonic function with boundary values $0$ on
$\partial \Omega$ and a singularity equivalent to $ \log \|z\|^2$.

We proceed with the proof that we can find an ``independent" pair of $f_i$'s.
\vskip.3cm

{\bf Case 1}: For any three point subset $\Tilde{S}_\varepsilon \subset S_\varepsilon$, 
the set of limit directions satisfies $\#\Tilde{\mathcal{D}} = 3$. So whenever 
$\{i,j\}$ and $\{i',j'\}$ have an element in common, $l_{ij}$ is independent from $l_{i'j'}$
and so for any $1\le k< k'\le 3$, $f_k$ and $f_{k'}$ have no common factor. So $\Psi_0^{-1}(0) = \{0\}$.
\vskip.3cm

{\bf Case 2}: Suppose that there exists a three point subset 
$S'_\varepsilon \subset S_\varepsilon$ such that the set $\mathcal{D}' $ of limit directions 
satisfies $\#\mathcal{D}' = 2$. 
Without loss of generality, $S'_\varepsilon = \{a_1^\varepsilon, a_2^\varepsilon, a_3^\varepsilon\} \subset S_\varepsilon$. 

Write $v_{ij}$ for the direction in $\mathbb P^1$ defined by $l_{ij}$. 
With our hypothesis, we may assume
$v_{23} = v_{12} \not= v_{13}$. 
It will be convenient to write
\begin{equation*}
\begin{aligned}
A_1 := \{v_{13}, v_{24}\} \cap \{v_{12}, v_{34}\},\\
A_2 := \{v_{13}, v_{24}\} \cap \{v_{14}, v_{23}\},\\
A_3 := \{v_{12}, v_{34}\} \cap \{v_{14}, v_{23}\}.
\end{aligned}
\end{equation*}
So here $A_3 \not= \emptyset$. We will show that there exists
$p \in \{1, 2\}$ such that $A_p = \emptyset$ (and thus the corresponding couple 
of function $f_i$ will be without a common factor, and the proof concluded). 

Suppose $A_1 \not= \emptyset$. Since $v_{23} = v_{12} \not= v_{13}$, by
\eqref{condition20092012},  $v_{12} \not= v_{24}$. 
Consequently,  $v_{34}\in\{ v_{13}, v_{24}\}$. 
%On va en d\'eduire que $A_2 = \emptyset$. Et alors $\{f_1, f_3\}$ sera ind\'ependant.

%On rappelle que $A_2 = \{v_{13}, v_{24}\} \cap \{v_{14}, v_{23}\}$. 
We study $A_2$. Since $v_{23}= v_{12} \not= v_{24}$, $v_{23} \notin \{v_{13}, v_{24}\}$.  
So we need to study $v_{14}$.
\vskip.3cm

{\bf Case 2.1}: $v_{34} = v_{13}$. 

Then (\ref{condition21112011}) implies that $v_{14} \not= v_{13}=v_{34}$. 
We will see that $v_{14} = v_{24}$ is impossible.
For this, we need to take some coordinates. 

Using  translations, we may assume $a_1^\varepsilon = 0 \in \mathbb{D}^2$, for any $\varepsilon$. 
Choose vectors $\Tilde{v}_{ij} \in \mathbb{C}^2$ such that $||\Tilde{v}_{ij}|| = 1$ 
and $[\Tilde{v}_{ij}] = v_{ij} \in \mathbb{P}^1\mathbb{C}, 1 \leqslant i < j \leqslant 4$. 
Since $v_{23} = v_{12} \not= v_{13}$, we can choose an invertible linear map $\Phi$ 
such that $[\Phi(\Tilde{v}_{12})] = [1: 0]$, $[\Phi(\Tilde{v}_{13})] = [0: 1]$. 
So we can study $\Phi(S_\varepsilon)$, 
where 
\begin{multline*}
\Phi(a_1^\varepsilon) = b_1^\varepsilon= (0,0), \; \Phi(a_2^\varepsilon) = b_2^\varepsilon = (\rho_2(\varepsilon), \eta_2(\varepsilon)), 
\\
\Phi(a_3^\varepsilon) = b_3^\varepsilon = (\eta_3(\varepsilon), \rho_3(\varepsilon)), 
\; \Phi(a_4^\varepsilon) = b_4^\varepsilon = (\alpha(\varepsilon), \beta(\varepsilon))
\end{multline*}
in which all coordinates tend to $0$ and $\underset{\varepsilon \to 0}{\lim}\; \eta_j(\varepsilon) / \rho_j(\varepsilon) = 0, j = 2, 3$. 
We retain the notation $v_{ij} \in \mathbb{P}^1\mathbb{C}, 1 \leqslant i < j \leqslant 4$, and
 $v_{ij}:=\lim_\eps v_{ij}^\varepsilon$ where this last
 is the direction of the line through $b_i^\varepsilon$ and $b_j^\varepsilon$. Let
 \begin{equation*}
\gamma(\varepsilon) := \dfrac{\rho_3(\varepsilon) - \eta_2(\varepsilon)}{\eta_3(\varepsilon) - \rho_2(\varepsilon)},
\end{equation*}
then $v_{23}^\varepsilon = [1:\gamma(\varepsilon)]$. Since $v_{23} = v_{12} = [1:0]$, 
 $\underset{\varepsilon \to 0}{\lim}\; \gamma(\varepsilon) = 0$. Thus
\begin{equation*}
\underset{\varepsilon \to 0}{\lim}\; 
\frac{\rho_3(\varepsilon)}{\rho_2(\varepsilon)}  
= \underset{\varepsilon \to 0}{\lim} \frac{\gamma(\varepsilon) - \frac{\eta_2(\varepsilon)}{ \rho_2(\varepsilon)}}
{\gamma(\varepsilon)\cdot \frac{\eta_3(\varepsilon)}{\rho_3(\varepsilon) } - 1} = 0.
\end{equation*}

Assume now that $v_{14} = v_{24}$. Then 
$[1:0] = v_{12} \not= v_{14} = v_{24} \not= v_{34} = [0:1]$. 
Write $v_{14} =[1:\ell]$, i.e. $\beta/\alpha \to \ell \not= 0, \infty$. 
Consider $\rho_2/\alpha$. If $ \|\rho_2/\alpha\| \leqslant C_2 < \infty$, as $\varepsilon \to 0$ 
(or even along a subsequence $\varepsilon_k\to 0$), then
\begin{equation*}
\frac{\alpha - \eta_3}{\beta - \rho_3} = 
\frac{1 - \frac{\eta_3}{\rho_3}\cdot \frac{\rho_3}{\rho_2}\cdot \frac{\rho_2}{\alpha}}{\frac\beta\alpha - \frac{\rho_3}{\rho_2}\cdot \frac{\rho_2}\alpha}
 \to \ell \not= 0, \; \mbox{ as }\; \varepsilon \to 0.
\end{equation*}
This contradicts $
\lim_{\varepsilon\to0} [\alpha - \eta_3:\beta - \rho_3]=v_{34} = v_{13} = [0:1]$. 
Therefore we have $\alpha/\rho_2 \to 0$, so
\begin{equation*}
\dfrac{\beta - \eta_2}{\alpha - \rho_2} = 
\frac{\frac\beta\alpha\cdot \frac\alpha{\rho_2} - \frac{\eta_2}{\rho_2}}
{\frac\alpha{\rho_2} - 1} \to 0, \; \mbox{as}\; \varepsilon \to 0.
\end{equation*}
This contradicts $\lim_{\varepsilon\to0} [\alpha - \rho_2:\beta - \eta_2]=v_{24} = v_{14} \not= v_{12} = [1:0]$. This is the contradiction we sought.
\vskip.3cm

{\bf Case 2.2}: $v_{34} = v_{24}$.

In an analogous way, we will see that 
$A_2 =   \emptyset$. 
We still have $v_{23} \notin \{v_{13}, v_{24}\}$. 
By condition  (\ref{condition20092012}), 
on a $v_{14} \not= v_{24} = v_{34}$. We still use the coordinates above.

Suppose that $v_{14} = v_{13} = [0:1]$. 
This implies $\alpha/\beta \to 0$. If $0 < \|\rho_2/\beta\| \leqslant C_4 < \infty$ as 
$\varepsilon \to 0$, 
\begin{equation*}
\dfrac{\alpha - \eta_3}{\beta - \rho_3} = \dfrac{\alpha/\beta - \eta_3/\rho_3. \rho_3/\rho_2. \rho_2/\beta}{1 - \rho_3/\rho_2. \rho_2/\beta} \to 0, \; \mbox{as}\; \varepsilon \to 0.
\end{equation*}
This contradicts $v_{34} = v_{24} \not= v_{23} = [0:1]$. Thus $\beta/\rho_2 \to 0$,
therefore
\begin{equation*}
\dfrac{\beta - \eta_2}{\alpha - \rho_2} = \dfrac{\beta/\rho_2 - \eta_2/\rho_2}{\alpha/\beta .\beta/\rho_2 - 1} \to 0, \; \mbox{as}\; \varepsilon \to 0.
\end{equation*}
This contradicts $v_{24} = v_{34} \not= v_{23} = [1:0]$. So $v_{14} \not= v_{13}$. 

\vskip.3cm

In a similar way, we can prove that if $A_2 \not= \emptyset$, then $A_1 = \emptyset$.
\hfill $\square$
\vskip.3cm
To finish the proof of Theorem \ref{theo_ICU_chap_4}, we need to prove the statements about the
limit ideal.  General properties of convergence show that $\ell (\mathcal I)=4$ and
the form of the generators show that $\mathcal I \subset \frak M_0^2$. It remains to
prove that $\mathcal I \supset \frak M_0^3$, which is a consequence of a more general fact.

\begin{prop}\label{pro1201}
Suppose that all the directions in $\mathcal{D}(S_\varepsilon)$ admit a limit,
and that  $\#\mathcal{D} \geqslant 2$. Then
 $\mathfrak{M}_0^3 \subset \underset{\varepsilon \to 0}{\lim}\inf \mathcal{I}_\varepsilon$. Furthermore,
$$
\limsup_\eps \mathcal{I}_\varepsilon \subset \mathfrak{M}_0^2.
$$
\end{prop}
\begin{proof}
%Tout d'abord, comme $\mathfrak{M}_0^3 = \left\langle z_1^3, z_1^2 z_2, z_1 z_2^2, z_2^3 \right\rangle$, 
%il suffira de montrer que chaque g\'en\'erateur de $\mathfrak{M}_0^3$ est approximable par des fonctions de $\mathcal{I}_\varepsilon$. 
$\mathfrak{M}_0^3$ is invariant under invertible linear maps. 
Since $\#\mathcal{D} \geqslant 2$, there exists $i \in \{1, 2, 3, 4\}$ such that $v_{ik} \not= v_{ik'}$, with $k \not= k'$ and $k, k' \in \{1, 2, 3, 4\} \backslash \{i\}$; otherwise it is easy to show that
all directions are equal, in contradiction with the hypothesis.

Without loss of generality, assume $v_{12}\neq v_{13}$
and after a linear transformation, $v_{12} = [1:0], v_{13} = [0:1]$. 

We reduce ourselves by translations to the case $a_1^\varepsilon =(0,0)$.
Let $a_2^\varepsilon  = (\rho_2(\varepsilon), \delta_2(\varepsilon))$ and $a_3^\varepsilon - a_1^\varepsilon = (\delta_3(\varepsilon), \rho_3(\varepsilon))$, where $\delta_j(\varepsilon) = o(\rho_j(\varepsilon))$, $j = 2, 3$. Let $a_4^\varepsilon  = (x_4(\varepsilon), y_4(\varepsilon))$ 
tending to $(0,ˆ)$. 
For any $\varepsilon$, set 
\begin{equation*}
\begin{aligned}
\psi_1^\varepsilon &:= \big[z_1 - x_1(\varepsilon) - \dfrac{\delta_3(\varepsilon)}{\rho_3(\varepsilon)}(z_2 - x_2(\varepsilon))\big]\big[z_1- x_1(\varepsilon) - \rho_2(\varepsilon)\big]
\big[z_1- x_1(\varepsilon) - x_4(\varepsilon)],
\\
\psi_2^\varepsilon &:=\big[z_1 - x_1(\varepsilon) - \dfrac{\delta_3(\varepsilon)}{\rho_3(\varepsilon)}(z_2 - x_2(\varepsilon))\big]\big[z_1- x_1(\varepsilon) - \rho_2(\varepsilon)\big]
\big[z_2 - x_2(\varepsilon) - y_4(\varepsilon)\big],
\\
\psi_3^\varepsilon &:=\big[z_1 - x_1(\varepsilon) - \dfrac{\delta_3(\varepsilon)}{\rho_3(\varepsilon)}(z_2 - x_2(\varepsilon))\big]\big[z_2 - x_2(\varepsilon) - \dfrac{\delta_2(\varepsilon)}{\rho_2(\varepsilon)}(z_1 - x_1(\varepsilon))\big]
\big[z_2 - x_2(\varepsilon) - y_4(\varepsilon)\big]
\\
\psi_4^\varepsilon &:= \big[z_2 - x_2(\varepsilon) - \dfrac{\delta_2(\varepsilon)}{\rho_2(\varepsilon)}(z_1 - x_1(\varepsilon))\big]\big[z_2- x_2(\varepsilon) - \rho_3(\varepsilon)\big]
\big[z_2 - x_2(\varepsilon) - y_4(\varepsilon)\big].
\end{aligned}
\end{equation*}
Then $\psi_j^\varepsilon \in \mathcal{I}_\varepsilon, 1 \leqslant j \leqslant 4$, %et pour tout $\varepsilon \in E$, 
and, with uniform convergence on compacta of $\Omega$,
\begin{equation*}
\begin{aligned}
z_1^3 &= \underset{\varepsilon \to 0}{\lim} \psi_1^\varepsilon \in \underset{\varepsilon \to 0}{\lim}\inf \mathcal{I}_\varepsilon\\
z_1^2z_2 &= \underset{\varepsilon \to 0}{\lim} \psi_2^\varepsilon \in \underset{\varepsilon \to 0}{\lim}\inf \mathcal{I}_\varepsilon\\
z_1z_2^2 &= \underset{\varepsilon \to 0}{\lim} \psi_3^\varepsilon \in \underset{\varepsilon \to 0}{\lim}\inf \mathcal{I}_\varepsilon\\
z_2^3 &= \underset{\varepsilon \to 0}{\lim} \psi_4^\varepsilon \in \underset{\varepsilon \to 0}{\lim}\inf \mathcal{I}_\varepsilon.
\end{aligned}
\end{equation*}
Thus $\mathfrak{M}_0^3 =\left\langle z_1^3, z_1^2 z_2, z_1 z_2^2, z_2^3 \right\rangle \subset \underset{\varepsilon \to 0}{\lim}\inf \mathcal{I}_\varepsilon$.

To get the other inclusion, we make the same normalizations (using the fact that 
$\mathfrak{M}_0^2$ is invariant under invertible linear transformations, too). 
Write $\Tilde{S}_\varepsilon = \{a_1^\varepsilon, a_2^\varepsilon, a_3^\varepsilon\}$. 
By   \cite[Theorem 1.12, i]{MRST}, 
$\underset{\varepsilon \to 0}{\lim}\; \mathcal{I}\big(\Tilde{S}_\varepsilon\big) = \mathfrak{M}_0^2$. 
Since $\mathcal{I}_\varepsilon \subset \mathcal{I}\big(\Tilde{S}_\varepsilon\big)$,$
\underset{\varepsilon \rightarrow 0}{\lim\sup}\;\mathcal{I}_\varepsilon \subset \underset{\varepsilon \rightarrow 0}{\lim\sup}\;\mathcal{I}\big(\Tilde{S}_\varepsilon\big) = \mathfrak{M}_0^2.$
\end{proof}

\subsection{Proof of Theorem \ref{theo_02_chap_04}}

The fact that the limit inferior of the Green functions is greater than the Green
function of the ideal, but not equal to it, follows from Theorem \ref{thmiff} 
since here $\mathcal I_0$ has $3$ generators. 

{\bf Remark.} It would be desirable to have a better estimate of the limits of Green functions.
Some explicit computations were carried out in \cite[Section 4.3]{Hai}, using the methods
from \cite{Ra-Th}.  It concerned the family of poles given by 
$S_\varepsilon := \{(0;0), (\varepsilon; 0), (0; \varepsilon), (\gamma\varepsilon;0)\}$, 
with $\gamma \not= 1$.  Since the family is homogeneous in $\epsilon$, in 
particular is given by a hyperplane section of a (singular) holomorphic curve,
\cite[Example 5.8]{Ra-Th} shows that the limit of the Green functions does exist.

The following estimates are obtained:
\begin{enumerate}
\item
$\underset{\varepsilon \to 0}{\lim}\; G_{\mathcal{I}_\varepsilon}(z) \geqslant 2\log \|z\| + O(1)$, for  $z_2\neq 0$;
\item
$\underset{\varepsilon \to 0}{\lim}\; G_{\mathcal{I}_\varepsilon}(z) \geqslant \frac{5}{3}\log \|z\| + O(1)$, for $z_1z_2^2(z_1 + z_2)(z_1 + \gamma z_2)\neq 0$.
\end{enumerate}
This is far from a complete answer, even in this case, but the computations involved 
are getting increasingly tedious.

\vskip.3cm
We now proceed with the proof of convergence of the family of ideals.

As before, we may assume $a_1^\varepsilon = 0 \in \Omega$. Since $\#\Tilde{\mathcal{D}} \geqslant 2$, for any three-point set $\Tilde{S}_\varepsilon \subset S_\varepsilon$, 
 $\# \mathcal{D} \geqslant 2$.
Without loss of generality, assume $v_{12} \not= v_{13}$. By (\ref{eq_15112012_02}), 
we may assume that for $i = 2$, $v_{12} = v_{23} = v_{24}$. 

Then we claim that $\# \mathcal{D} \geqslant 3$. Indeed, if we had $\# \mathcal{D} = 2$, then $\mathcal{D} =\{v_{12} , v_{13}\}$.
Three cases may occur.

$\bullet$) If $v_{14} = v_{12}$, then $v_{12} = v_{14} = v_{24}$. This contradicts \eqref{condition21112011}.

$\bullet$) If $v_{34} = v_{12}$, then $v_{23} = v_{34} = v_{24}$. This contradicts \eqref{condition21112011}.

$\bullet$) Si $v_{14} = v_{34} = v_{13}$, then the this contradicts \eqref{condition21112011}.

This proves the claim. 

We can chose an invertible linear map $\Phi : \mathbb{C}^2 \rightarrow \mathbb{C}^2$ such that
\begin{equation*}
[\Phi(\Tilde{v}_{12})] = [1:0]\; \mbox{and} \; [\Phi(\Tilde{v}_{13})] = [0:1],
\end{equation*}
where $\Tilde{v}_{12}, \Tilde{v}_{13} \in \mathbb{C}^2$ are chosen so that $\|\Tilde{v}_{12}\| = \|\Tilde{v}_{13}\| = 1$ and
$[\Tilde{v}_{12}] = v_{12}$, $[\Tilde{v}_{13}] = v_{13}$. Then
\begin{equation*}
\Phi\big(S_\varepsilon\big) = S'_\varepsilon = \{b_1^\varepsilon = (0, 0), b_2^\varepsilon, b_3^\varepsilon, b_4^\varepsilon\}.
\end{equation*}
For this new system $v_{12} = [1:0] \not= v_{13} = [0:1]$. We can choose $l_{ij}^\varepsilon(z)$,
normalized equations of the lines through the pairs of
points $b_i^\varepsilon$ and $b_j^\varepsilon$, $1 \leqslant i < j \leqslant 4$
such that $\underset{\varepsilon \to 0}{\lim}\;l_{12}^\varepsilon(z) = \underset{\varepsilon \to 0}{\lim}\;l_{23}^\varepsilon(z) = \underset{\varepsilon \to 0}{\lim}\;l_{24}^\varepsilon(z) = z_2$ 
and $\underset{\varepsilon \to 0}{\lim}\;l_{13}^\varepsilon(z) = z_1$. This implies
\begin{equation*}
z_1z_2 = \underset{\varepsilon \to 0}{\lim}\; l_{13}^\varepsilon(z) l_{24}^\varepsilon(z)
 \in \liminf_\eps\mathcal{I}_\eps,
\end{equation*}
\begin{equation*}
z_1^3 = \underset{\varepsilon \to 0}{\lim}\; l_{13}^\varepsilon(z) \big[z_1 - z_1(b_2^\varepsilon)\big]\big[z_1 - z_1(b_4^\varepsilon)\big] \in \liminf_\eps\mathcal{I}_\eps.
\end{equation*}
So $\left\langle z_1z_2, z_1^3 \right\rangle \subset \mathcal{I}_\ast$. 

Since $\# \mathcal{D} \geqslant 3$, there exists $(i,j) \in \{ (1,3), (1,4), (3,4)\}$ 
such that $v_{ij}^\varepsilon \to [1:t]$, 
with $t \not= 0, \infty$. So $\underset{\varepsilon \to 0}{\lim}\; l_{ij}^\varepsilon(z) = z_2 - t z_1 \in \liminf_\eps\mathcal{I}_\eps$. This implies
\begin{equation*}
z_2^2 = \underset{\varepsilon \to 0}{\lim}\; \big(l_{ij}^\varepsilon(z) l_{km}^\varepsilon(z) + t l_{13}^\varepsilon(z) l_{24}^\varepsilon(z)\big) \in \liminf_\eps\mathcal{I}_\eps,
\end{equation*}
since $2 \in \{k, m \} = \{1,2,3,4\} \setminus \{i,j\}$. Thus $\mathcal{I}_0 := \left\langle z_1z_2, z_2^2, z_1^3\right\rangle \subset \liminf_\eps\mathcal{I}$, with $\ell\big(\mathcal{I}_0\big) = 4$. 
By Theorem \ref{limsupinf}, $\underset{\varepsilon \to 0}{\lim}\; \mathcal{I}\big(S_\varepsilon\big) = \mathcal{I}_0$.  \hfill $\square$

\subsection{Proof of Theorem \ref{pro11122012_01}}

By the hypothesis $\# \mathcal{D} \geqslant 2$, we may assume $v_{12} \not= v_{13}$.
 Just as in the proof of Theorem \ref{theo_02_chap_04}, we perform a translation
 to reduce ourselves to $a_1^\eps=(0,0)$, and we choose a linear map $\Phi$ so that 
 we are reduced to $v_{12} = [1:0] \not= v_{13} = [0:1]$.  We adopt the
 same notation $S'_\varepsilon=\{b_k^\varepsilon, 1\le k \le 4\}$.
 
Since there is a $3$ point subset $\tilde{S'}_\varepsilon \subset S'_\varepsilon$ 
such that  $\# \Tilde{\mathcal{D}}' = 1$, we may 
 assume that $\tilde S'_\eps=\{1,2,4\}$, so $v_{12} = v_{14} = v_{24} = [1:0]$.
 Again we may choose line equations so that $\underset{\varepsilon \to 0}{\lim}\;l_{12}^\varepsilon(z) = \underset{\varepsilon \to 0}{\lim}\;l_{14}^\varepsilon(z) = \underset{\varepsilon \to 0}{\lim}\;l_{24}^\varepsilon(z) = z_2$ 
and $\underset{\varepsilon \to 0}{\lim}\;l_{13}^\varepsilon(z) = z_1$. 

The proof in case (i) can then be completed exactly as the proof of Theorem \ref{theo_02_chap_04}
above.

\vskip.3cm
Case (ii):  $\# \mathcal{D} = 2$.

Either there exists $(i,j) \in \{(2,3), (3,4)\}$ such that $v_{ij}^\varepsilon \to v_{ij} = [1:0]$ or $v_{23} = v_{34} = [0:1]$.

Case (ii.1): there exists $(i,j) \in \{(2,3), (3,4)\}$ such that $v_{ij}^\varepsilon \to v_{ij} = [1:0]$.

Then $\underset{\varepsilon \to 0}{\lim}\; l_{ij}^\varepsilon(z) = z_2$. Then, again, 
\begin{equation*}
z_2^2 = \underset{\varepsilon \to 0}{\lim}\; l_{ij}^\varepsilon(z) l_{km}^\varepsilon(z) \in 
\liminf_\eps \mathcal{I}_\eps,
\end{equation*}
since $2 \in \{k, m \} = \{1,2,3,4\} \setminus \{i,j\}$. 
Again, as in the proof of Theorem \ref{theo_02_chap_04}, we find that  
$\underset{\varepsilon \to 0}{\lim}\; \mathcal{I}\big(S_\varepsilon\big) = \mathcal{I}_0$.

Case (ii.2): $v_{23} = v_{34} = [0:1]$. 

Thus $v_{13} = v_{23} = v_{34} = [0:1]$ and $v_{12} = v_{14} = v_{24} = [1:0]$. 

As in \cite{Du-Th}, where systems of three points tending to the origin
along a single direction are considered, 
we reparametrize  $\{b_1^\varepsilon,b_2^\varepsilon,b_4^\varepsilon\}$ in such a way that 
$|\varepsilon| = \|b_2^\varepsilon - b_1^\varepsilon\|$ 
and choose a coordinate system depending on $\eps$ such that 
$$
b_1^\varepsilon = (0, 0), b_2^\varepsilon = (\varepsilon, 0), b_4^\varepsilon 
= \big(\rho(\varepsilon), \delta(\varepsilon)\rho(\varepsilon)\big)
\mbox{ where }0 < |\rho(\varepsilon)| \leqslant \frac{1}{2}|\varepsilon|, \delta(\varepsilon) \to 0,
$$
as $\varepsilon \to 0$. 
Denote $b_3^\varepsilon = \big(\alpha(\varepsilon), \beta(\varepsilon)\big)$.
Since $v_{13}^\varepsilon = [\alpha(\varepsilon):\beta(\varepsilon)] \to [0:1]$, 
 $\underset{\varepsilon \to 0}{\lim}\; \frac{\alpha(\varepsilon)}{\beta(\varepsilon)} = 0$. 
We will write $\rho = \rho(\varepsilon), \delta = \delta(\varepsilon), \alpha = \alpha(\varepsilon), \beta = \beta(\varepsilon)$. For $\delta$ small enough, set
\begin{equation*}
\Tilde{\delta} := \dfrac{\delta}{1 - \dfrac{\alpha}{\beta}\delta}, \quad
\Tilde{\rho} := \rho\big(1 - \dfrac{\alpha}{\beta} \delta\big).
\end{equation*}
Clearly $\Tilde{\delta}, \Tilde{\rho}\to 0$. Furthermore,
\begin{equation*}
\dfrac{\Tilde{\delta}}{\Tilde{\rho} - \varepsilon} = \dfrac{\delta/(\rho-\varepsilon)}{\bigg(1-\dfrac{\alpha}{\beta}\dfrac{\delta\rho}{\rho-\varepsilon}\bigg)\big(1-\dfrac{\alpha}{\beta}\delta\big)},
\end{equation*}
so if $\underset{\varepsilon \to 0}{\lim}\; \frac{\delta}{\rho - \varepsilon} = m$, then $\underset{\varepsilon \to 0}{\lim}\; \frac{\Tilde{\delta}}{\Tilde{\rho} - \varepsilon} = m$. 
Consider the following biholomorphism (a small perturbation of the identity map):
\begin{equation*}
\Phi_{1, \varepsilon} :  \mathbb{C}^2 \longrightarrow \mathbb{C}^2, \; z \mapsto \Phi_{1, \varepsilon}(z) = \bigg(z_1 - \dfrac{\alpha}{\beta} z_2, z_2\bigg),
\end{equation*}
Then $\Phi_{1, \varepsilon}\big(S'_\varepsilon\big) = S_{1, \varepsilon} = \big\{(0,0), (\varepsilon, 0), (\Tilde{\rho}, \Tilde{\delta}\Tilde{\rho}), (0, \beta)\big\}$. 
Since $v_{23} = [0:1]$ and $v_{13} = [0:1]$,  $|\alpha - \varepsilon| \ll \frac{1}{2}|\beta|$ et $|\alpha| \ll \frac{1}{2}|\beta|$. So $|\varepsilon| \leqslant |\alpha - \varepsilon| + |\alpha| \ll |\beta|$. 

The proof is concluded with the following result. 
Notice that this limit ideal in case (ii.2) is deduced from $\mathcal I_0$ by exchanging the
coordinates $z_1$ and $z_2$, so is again equivalent to it by a linear invertible map.

\begin{prop}\label{pro01_15112012} 
Let $S_\varepsilon = \{(0,0), (\varepsilon,0), (\rho, \delta\rho), (0, \beta)\}$ 
tend to $(0,0)$ as $\varepsilon\to 0$, with $\rho := \rho(\varepsilon), \delta := \delta(\varepsilon), \beta := \beta(\varepsilon)$ and $0 < |\rho| \leqslant \frac{1}{2}|\varepsilon|$, $|\varepsilon| \ll |\beta|$. Then

(i) \; If $\underset{\varepsilon \to 0}{\lim}\; \dfrac{\delta}{\rho - \varepsilon} = m \not= \infty$, 
$
\underset{\varepsilon \to 0}{\lim}\; \mathcal{I}\big(S_\varepsilon\big) = \mathcal{I}_0 := \left\langle z_1z_2, z_2^2, z_1^3\right\rangle.$

ii) \; If $\underset{\varepsilon \to 0}{\lim}\; \dfrac{\delta}{\varepsilon} = \infty$, 
we have two cases:
\begin{itemize}
\item[ii.1)]\;\;If $\underset{\varepsilon \to 0}{\lim}\; \dfrac{\rho - \varepsilon}{\delta\beta} = k \notin \{0, \infty\}$, then
$
\underset{\varepsilon \to 0}{\lim}\; \mathcal{I}\big(S_\varepsilon\big) = \mathcal{J}_0 := \left\langle z_1z_2, z_1^2 + k z_2^2, z_1^3\right\rangle.
$
\item[ii.2)] \;\;If $\underset{\varepsilon \to 0}{\lim}\; \dfrac{\rho - \varepsilon}{\delta\beta} = k \in \{0, \infty\}$, then $\lim_\eps \mathcal{I}\big(S_\varepsilon\big)=\mathcal{I}_0$, if $k = \infty$ and $\mathcal{I}_1 := \left\langle z_1z_2, z_1^2, z_2^3 \right\rangle$, if $k = 0$.
\end{itemize}
\end{prop}

\begin{proof}  Since $|\varepsilon| \ll |\beta|$,
\begin{equation*}
z_1z_2 = \underset{\varepsilon \to 0}{\lim}\; \big(z_2 - \rho z_1\big)\big[z_1 + \dfrac{\varepsilon}{\beta}z_2 - \varepsilon\big] \in \liminf_\eps \mathcal{I_\eps}, \mbox{ and}\\
\end{equation*}
\begin{equation*}
z_1^3 = \underset{\varepsilon \to 0}{\lim}\; z_1(z_1 - \rho)(z_1 - \varepsilon) \in \liminf_\eps \mathcal{I_\eps}.
\end{equation*}
Thus $\left\langle z_1z_2, z_1^3 \right\rangle \subset \liminf_\eps \mathcal{I_\eps}$. 

Now we need to look at various cases separately.
\vskip.3cm
i) Since $\underset{\varepsilon \to 0}{\lim}\; \dfrac{\delta}{\rho - \varepsilon} = m \not= \infty$ 
and the polynomial
\begin{equation*}
Q_\varepsilon(z) := \dfrac{\delta\varepsilon}{\rho - \varepsilon}(\delta\rho - \beta)z_1 - \beta z_2 - \dfrac{\delta}{\rho - \varepsilon}(\delta\rho - \beta) z_1^2 + z_2^2 \in \mathcal{I}\big(S_\varepsilon\big),
\end{equation*}
we obtain $z_2^2 = \underset{\varepsilon \to 0}{\lim}\;Q_\varepsilon(z) \in \liminf_\eps \mathcal{I_\eps}$. So 
$\mathcal{I}_0 := \left\langle z_1z_2, z_2^2, z_1^3\right\rangle \subset \liminf_\eps \mathcal{I_\eps}$. 
Since $\ell(\mathcal{I}_0) = 4$, applying Theorem \ref{limsupinf} $\underset{\varepsilon \to 0}{\lim}\; \mathcal{I}\big(S_\varepsilon\big) = \mathcal{I}_0$.

ii)  Since $0 < |\rho| \leqslant \frac{1}{2}|\varepsilon|$,
$\frac{|\varepsilon|}{2} \leqslant |\rho - \varepsilon| \leqslant |\varepsilon|$,
so if $\underset{\varepsilon \to 0}{\lim}\; \dfrac{\delta}{\varepsilon} = \infty$,
 $\underset{\varepsilon \to 0}{\lim}\;\dfrac{\rho - \varepsilon}{\delta} = \underset{\varepsilon \to 0}{\lim}\;\dfrac{\rho - \varepsilon}{\varepsilon}\dfrac{\varepsilon}{\delta} = 0$. 
 
 We consider two subcases:
 
ii.1) Suppose $\underset{\varepsilon \to 0}{\lim}\; \dfrac{\rho - \varepsilon}{\delta\beta} = k \notin \{0, \infty\}$. Consider the polynomial
\begin{equation}\label{eq_15112011_03}
P_\varepsilon(z) := - \varepsilon z_1 + \dfrac{\rho - \varepsilon}{\delta\beta} \dfrac{\beta}{\dfrac{\delta\rho}{\beta} - 1} z_2 + z_1^2 - \dfrac{\rho - \varepsilon}{\delta\beta} \dfrac{1}{\dfrac{\delta\rho}{\beta} - 1}z_2^2.
\end{equation}
We can check that $P_\varepsilon(z) \in \mathcal{I}\big(S_\varepsilon\big)$. 
Since $|\delta\rho| \ll |\rho| \leqslant \frac{1}{2}|\varepsilon| \ll |\beta|$, il vient $\frac{\delta\rho}{\beta} \to 0$. So
\begin{equation*}
z_1^2 + kz_2^2 = \underset{\varepsilon \to 0}{\lim}\; P_\varepsilon(z) \in
\liminf_\eps \mathcal{I_\eps}.
\end{equation*}
Thus $\mathcal{J}_0 := \left\langle z_1z_2, z_1^2 + kz_2^2, z_1^3\right\rangle 
\subset \liminf_\eps \mathcal{I_\eps}$. 
But the class $[z_1^2] = [z_1^2 + kz_2^2] - k[z_2^2] = - k[z_2^2] \in \mathcal{O}(\Omega) / \mathcal{J}_0$, thus 
$\mathcal{O}(\Omega) / \mathcal{J}_0 = Span \{ [1], [z_1], [z_2], [z_2^2] \}$
and $\ell(\mathcal{J}_0) = 4$. Using Theorem \ref{limsupinf}, we conclude $\underset{\varepsilon \to 0}{\lim}\; \mathcal{I}\big(S_\varepsilon\big) = \mathcal{J}_0$.

ii.2) Suppose $\underset{\varepsilon \to 0}{\lim}\; \dfrac{\rho - \varepsilon}{\delta\beta} = k \in \{0, \infty\}$. Analogously to \eqref{eq_15112011_03}, consider the polynomial
\begin{equation*}
R_\varepsilon(z) := \dfrac{\delta\beta}{\varepsilon}\big(\dfrac{\delta\rho}{\beta} - 1\big)\dfrac{\varepsilon}{\rho - \varepsilon}\varepsilon z_1 - \beta z_2 - \dfrac{\delta\beta}{\varepsilon}\big(\dfrac{\delta\rho}{\beta} - 1\big)\dfrac{\varepsilon}{\rho - \varepsilon} z_1^2 + z_2^2.
\end{equation*}
We can check that $R_\varepsilon(z) \in \mathcal{I}\big(S_\varepsilon\big)$. If $k = \infty$, 
then $|\delta\beta| \ll |\rho - \varepsilon| \ll |\delta|$, and
\begin{equation*}
\underset{\varepsilon \to 0}{\lim}\; \dfrac{\delta\beta}{\varepsilon} = \underset{\varepsilon \to 0}{\lim}\; \dfrac{\delta\beta}{\rho - \varepsilon}\dfrac{\rho - \varepsilon}{\varepsilon} = 0.
\end{equation*}
Thus $z_2^2 = \underset{\varepsilon \to 0}{\lim}\; P_\varepsilon(z) \in \liminf_\eps \mathcal{I}_\eps$.
Then $\mathcal{I}_0 \subset \liminf_\eps \mathcal{I}_\eps$. Since $\ell(\mathcal{I}_0) = 4$, using Theorem \ref{limsupinf}, we conclude $\underset{\varepsilon \to 0}{\lim}\; \mathcal{I}\big(S_\varepsilon\big) = \mathcal{I}_0$.

Finally, if $k = 0$, $|\rho - \varepsilon| \ll |\delta\beta| \ll |\delta|$. From \eqref{eq_15112011_03}
we deduce
$z_1^2 = \underset{\varepsilon \to 0}{\lim}\; P_\varepsilon(z) \in \liminf_\eps \mathcal{I}_\eps.
$
In addition,
$ z_2^3 = \underset{\varepsilon \to 0}{\lim}\; z_2(z_2 - \delta\rho)(z_2 - \beta) \in \liminf_\eps \mathcal{I}_\eps$.

Therefore $\mathcal{I}_1 := \left\langle z_1z_2, z_1^2, z_2^3 \right\rangle \subset \liminf_\eps \mathcal{I}_\eps$. We conclude as before.
\end{proof}

{}


\begin{thebibliography}{}
\bibitem{De1} J.-P. Demailly, {\it Mesures de Monge-Amp\`ere et mesures
pluriharmoniques}, Math. Z. 194 (1987), 519--564.

\bibitem{Du-Th} Duong Quang Hai, P. J. Thomas, {\it Limit Of Three-Point Green Functions : The Degenerate Case}, Serdica 40 (2014), 99--110. 

\bibitem{Hai} Duong Quang Hai, {\it Limites d'id\'eaux de fonctions holomorphes et de fonctions
de Green pluricomplexes}, Ph. D. thesis, Universit\'e Toulouse III Paul Sabatier, July 8th, 2013, 94 pp.

\bibitem{Ho} L. H\"ormander, {\it An Introduction to Complex Analysis in 
Several Variables}, Third Edition (revised), Mathematical Library, Vol. 7,
North Holland, Amsterdam-New York-Oxford-Tokyo, 1990.

\bibitem{Lel} P. Lelong, {\it Fonction de Green pluricomplexe et lemmes
de Schwarz dans les espaces de Banach}, J. Math. Pures Appl. 68 (1989), 319--347.

\bibitem{Lem} L. Lempert, {\it La m\'etrique de Kobayashi et la 
repr\'esentation des domaines sur la boule}, Bull. Soc. Math. France 109 (1981), no. 4, 427--474. 

\bibitem{MRST} J. I. Magnusson, A. Rashkovskii, R. Sigurdsson, P. J. Thomas,  {\it Limits of multipole pluricomplex Green functions}, Int. J. Math. 23 (2012), no. 6, DOI: 10.1142/S0129167X12500656.

\bibitem{Ra-Si} A. Rashkovskii, R. Sigurdsson, {\it Green functions with singularities along complex spaces},
Int. J. Math. 16 (2005), no. 4, 333--355.

\bibitem{Ra-Th} A. Rashkovskii, P. J. Thomas, 
 {\it Powers of ideals and convergence of Green functions with colliding poles}, 
 Int. Math. Res. Not. IMRN (2014) 1253--1272.
 
 \bibitem{Th} P. J. Thomas, {\it Green vs. Lempert functions: a minimal example}, 
 Pacific J. Math., Vol. 257 (2012), no. 1, pp. 189--197. 

\bibitem{Ts} A. K. Tsikh, {\it Multidimensional Residues and Their 
Applications}, Translations of Mathematical Monographs, Vol. 103,
American Mathematical Society, Providence, 1992.

\bibitem{Za} {V.P. Zahariuta}, {\it Spaces of analytic functions and maximal
plurisubharmonic functions.} D. Sci. Dissertation, Rostov-on-Don,
1984.

\end{thebibliography}
\end{document}